\documentclass{amsart}%[12pt,a4paper]{scrartcl}

\usepackage[utf8]{inputenc}
\usepackage[T1]{fontenc}

\usepackage[alphabetic]{amsrefs}
\usepackage{hyperref, cleveref}

\BibSpec{arXiv}{
	+{}{\PrintAuthors}{author}
	+{,}{ \textit}{title}
	+{}{ \parenthesize}{date}
	+{,}{ arXiv }{eprint}
}

\usepackage{subcaption}
\usepackage[subrefformat=parens]{caption}
\usepackage{enumitem}
\setlist[enumerate,1]{ref=(\arabic*)}

\usepackage{mathtools, amsmath, amssymb, amsthm}

\theoremstyle{plain}
\newtheorem{Theorem}{Theorem}[section]
\newtheorem{Lemma}[Theorem]{Lemma}
\newtheorem{TheoremA}{Theorem}

\theoremstyle{definition}
\newtheorem{Definition}[Theorem]{Definition} 
\newtheorem{Remark}[Theorem]{Remark}
\newtheorem{Example}[Theorem]{Example}

\usepackage{tikz, wasysym}
\usetikzlibrary{automata,arrows,positioning,decorations.pathreplacing,calc}

\definecolor{yellow}{RGB}{221, 170, 51}
\definecolor{red}{RGB}{187, 85, 102}
\definecolor{blue}{RGB}{0, 68, 136}

\newcommand{\N}{\mathbb{N}}
\newcommand{\R}{\mathbb{R}}
\newcommand{\Z}{\mathbb{Z}}

\newcommand{\ord}{\mathrm{ord}}

\newcommand{\bns}[1]{\Sigma^1(#1)}
\newcommand{\obns}[1]{\Sigma^1_\ord(#1)}
\newcommand{\relS}[2]{S(#1, #2)}
\newcommand{\orelS}[2]{S_\ord(#1, #2)}

\newcommand{\grp}[2]{{\langle #1 \mid #2 \rangle}}
\newcommand{\bgrp}[2]{{\Bigl\langle #1 \bigm\vert #2 \Bigr\rangle}}
\newcommand{\grpg}[1]{{\langle #1 \rangle}}

\makeatletter
\newsavebox{\@brx}
\newcommand{\llangle}[1][]{\savebox{\@brx}{\(\m@th{#1\langle}\)}%
  \mathopen{\copy\@brx\mkern2mu\kern-0.9\wd\@brx\usebox{\@brx}}}
\newcommand{\rrangle}[1][]{\savebox{\@brx}{\(\m@th{#1\rangle}\)}%
  \mathclose{\copy\@brx\mkern2mu\kern-0.9\wd\@brx\usebox{\@brx}}}
\makeatother
\newcommand{\ngrpg}[1]{{\llangle #1 \rrangle}}

\newcommand{\Cay}[1]{\operatorname{Cay}(#1)}

\newcommand{\Hom}{\operatorname{Hom}}
\newcommand{\rk}{\operatorname{rk}}

\newcommand{\ab}{\mathrm{ab}}
\renewcommand{\land}{~\mathrm{and}~}
\renewcommand{\lor}{~\mathrm{or}~}

\renewcommand{\leq}{\mathop{\leqslant}} %TODO correct spacing

\renewcommand{\geq}{\mathop{\geqslant}}

\newcommand{\inv}[1]{{#1}^{-1}}

\newcommand{\restr}[2]{\mathrel{{#1}_{|{#2}}}}

\newcommand{\comm}[2]{{[{#1},{#2}]}}

\newcommand{\onto}{\twoheadrightarrow}
\newcommand{\into}{\hookrightarrow}

\newcommand{\ordl}{\mathrel{\prec}}
\newcommand{\ordle}{\mathrel{\preccurlyeq}}
\newcommand{\ordnl}{\mathrel{\nprec}}

\newcommand{\ordg}{\mathrel{\succ}}
\newcommand{\ordge}{\mathrel{\succcurlyeq}}

\newcommand{\indol}[2]{\mathrel{#1^{#2}}}
\newcommand{\indor}[2]{\mathrel{#1^{\inv{#2}}}}

\DeclareFontFamily{U}{mathb}{\hyphenchar\font45}
\DeclareFontShape{U}{mathb}{m}{n}{
	<-6> mathb5 <6-7> mathb6 <7-8> mathb7
	<8-9> mathb8 <9-10> mathb9
	<10-12> mathb10 <12-> mathb12
}{}
\DeclareSymbolFont{mathb}{U}{mathb}{m}{n}
\DeclareMathSymbol{\ordll}{\mathrel}{mathb}{"CE}
\DeclareMathSymbol{\ordgg}{\mathrel}{mathb}{"CF}

\newcommand{\posconex}[2]{{#1}^{#2}}
\newcommand{\poscone}[1]{\posconex{#1}{\ordg}}

\usepackage{stackengine, scalerel}
\stackMath
\def\normsub{\ThisStyle{\mathrel{%
  \stackinset{r}{.75pt+.15\LMpt}{t}{.1\LMpt}{\rule{.3pt}{1.1\LMex+.2ex}}{\SavedStyle\leqslant}%
}}}
\newcommand{\subg}{\leqslant}
\newcommand{\supg}{\geqslant}

\begin{document}
\title{Sigma invariants for partial orders on nilpotent groups}
\author{Kevin Klinge}
\address{Faculty of Mathematics, Karlsruhe Institute of Technology, Englerstraße 2, 76131 Karlsruhe, Germany}
\email{kevin.klinge@kit.edu}

\begin{abstract}
	We prove that a map onto a nilpotent group \(Q\) has finitely generated kernel if and only if the preimage of the positive cone is coarsely connected as a subset of the Cayley graph for every full archimedean partial order on \(Q\).
	In case \(Q\) is abelian, we recover the classical theorem that \(N\) is finitely generated if and only if \(S(G,N) \subseteq \bns G\).
	Furthermore, we provide a way to construct all such orders on nilpotent groups.
	A key step is to translate the classical setting based on characters into a language of orders on \(G\).
\end{abstract}

\maketitle

\section{Introduction}

Let \(G\) be a finitely generated group and \(N \normsub G\) a normal subgroup.
A natural question to ask is in which cases \(N\) is also finitely generated.
A well known answer if \(G/N\) is abelian lies in the \(\Sigma\)-invariant via the following theorem.

\begin{Theorem}[Bieri, Neumann, Strebel \cite{BNS}]\label{intro-bns-iff-fg}
	Let \(G\) be a finitely generated group and \(N \normsub G\) a normal subgroup such that \(G/N\) is abelian.
	Then \(N\) is finitely generated if and only if \(\Phi \in \bns G\) for any character \mbox{\(\Phi\colon G \to \R\)} that factors through \(G/N\).
\end{Theorem}

By \emph{characters}, we mean group homomorphisms to \((\R, +)\), up to multiplication by a positive number.
\(\bns G\) is the set of characters such that the full subgraph of the Cayley graph of \(G\) spanned by \(\inv\Phi([0,\infty))\) is connected.
We denote the set of characters that factor through \(G/N\) by \(\relS G N\). It is also called the \emph{relative character sphere}.
In this notation, \Cref{intro-bns-iff-fg} asks if \(\relS G N \subseteq \bns G\).

We answer the question of finite generatedness in the larger case \(G/N\) nilpotent by introducing \(\obns G\) and \(\orelS G N\) to be sets of certain partial orders on \(G\).

\begin{TheoremA}[\Cref{bns-iff-fg}]\label{intro-bns-iff-fg-nilpotent}
	Let \(G\) be a finitely generated group and \(N \normsub G\) a normal subgroup such that \(G/N\) is nilpotent.
	Then \(N\) is finitely generated if and only if \(\orelS G N \subseteq \obns G\).
\end{TheoremA}

Here, \(\obns G\) contains those orders such that for any maximal normal antichain subgroup \(K \normsub G\), the subset of \(G\) above \(K\) is coarsely connected.
\(\orelS G N\) are \emph{full archimedean} orders, such that \(N\) is an antichain.
In case \(G/N\) is abelian, these are in 1-to-1 correspondence with \(\bns G\) and \(\relS G N\) respectively.

To see why we need this transition to orders, an important observation is that any character \(G \to \R\) that factors through \(G/N\) also factors through the abelianisation of \(G/N\) as indicated in \Cref{diagram-char-factors-through-ab}.
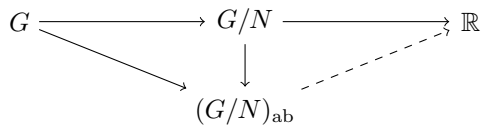
\begin{figure}
	\centering
\begin{tikzpicture}[shorten >=1pt,on grid,auto]

	\node[] (G) {\(G\)};
	\node[] (GN) [right=3 of G] {\(G/N\)};
	\node[] (R) [right=3 of GN] {\(\R\)};
	\node[] (ab) [below=1.2 of GN] {\((G/N)_\ab\)};
	\path[->]
		(G) edge (GN)
		(GN) edge (R)
		(G) edge (ab)
		(GN) edge (ab)
		(ab) edge [dashed] (R)
	;
\end{tikzpicture}
	\caption{Any map \(G \to \R\) whose kernel contains \(N\) factors through \((G/N)_\ab\).}
	\label{diagram-char-factors-through-ab}
\end{figure}
Hence \(G/N\) abelian is a necessary condition for \Cref{intro-bns-iff-fg}.
However, every character induces a partial order on \(G\) by pulling the standard order on \(\R\) back to \(G\).
We show that we can recognise orders induced by characters that factor through \(G/N\) by introducing the notion of \emph{full archimedean} orders on \(G\) such that \(N\) is an antichain.
Using this correspondence, \Cref{intro-bns-iff-fg} is the same as \Cref{intro-bns-iff-fg-nilpotent} in the special case where \(G/N\) is abelian.

In the more general case \(G/N\) nilpotent, full archimedean orders and characters are no longer in 1-to-1 correspondence.
We provide a similar description of those orders in the nilpotent case.
Still, every character on \(G\) induces an order.
But characters defined on subgroups also induce orders that are full archimedean.
The following theorem makes this relation precise.
\begin{TheoremA}[\Cref{characterisation-nilpotent-orders}]\label{intro-classification-nilpotent-orders}
	Let \(Q\) be a nilpotent group and \(\ordl\) a full archimedean partial order on \(Q\).

	Then there exists a normal subgroup \(P \normsub Q\) and an injective character on the center \mbox{\(\iota\colon Z(Q/P) \into \R\)} such that \[
		1 \ordl g ~\text{if and only if}~ gP ~\text{is central in}~ Q/P ~\text{and}~ \iota(gP) > 0.
	\]
\end{TheoremA}
We will also see that every such character induces a full archimedean order on \(G\), so this is a complete characterisation of full archimedean orders.

\bigskip
In \cite{Kielak}, Kielak discusses the question if \(G\) fibres. That is, if there exists a finitely generated normal subgroup \(N \normsub G\) such that \(G/N \cong \Z\).
He provides an answer for a certain class of groups, namely if \(G\) is virtually RFRS.
A group that is both RFRS and nilpotent is necessarily abelian.
In this sense, RFRS and nilpotent are orthogonal properties.
Thus, nilpotent groups are a natural class to study in the context of fibrations.
\Cref{intro-bns-iff-fg} plays an important role in Kielak's work.
To see if his proof can be adapted to provide an answer for a larger class of groups by using \Cref{intro-bns-iff-fg-nilpotent} instead will be the subject of future research.

In this work we concern ourselves only with the first \(\Sigma\)-invariant.
There are analogues of \(\bns G\) introduced in \cite{Renz} that correspond to higher finiteness properties of a character's kernel.
The result on fibring of \(G\) also holds analogously for higher \(\Sigma\)-invariants as was shown in \cite{Fisher}.
In this context it would be interesting to find analogues of higher \(\Sigma\)-invariants in the language of orders.

While we have seen that \Cref{intro-bns-iff-fg} is in general not true if \(G/N\) is non-abelian, there is no reason obvious to me why a restriction on \(G/N\) is strictly required for \Cref{intro-bns-iff-fg-nilpotent}.
The reason why we include the condition \(G/N\) nilpotent is because our proof heavily depends on it. We use the fact that nilpotent groups are relatively close to abelian groups to reduce \Cref{intro-bns-iff-fg-nilpotent} to the known abelian case.
It would be interesting to know if \Cref{intro-bns-iff-fg-nilpotent} is true for an even larger class of groups.

The authors in \cite{MNS} ask in which situations characters on a subgroup \(H \subg G\) extend to a character on \(G\) and how this behaves with respect to \(\Sigma\)-invariants.
We will study partial orders on subgroups and see how they extend to the whole group.
\Cref{intro-classification-nilpotent-orders} may be understood as a decomposition of the space of orders on a nilpotent group \(Q\) into character spheres of certain subgroups.
One important step in the proof of \Cref{intro-bns-iff-fg-nilpotent} will be to understand how this lifts to the space of orders of an arbitrary finitely generated group \(G\) that maps onto a nilpotent group.
This partially answers how the \(\Sigma\)-invariant and character sphere behave when passing to subgroups within the language of orders.

\bigskip
An important example to keep in mind is the case \(G/N \cong \Z\).
Then there are only two characters that factor through \(G/N\), namely the projection map and the projection map concatenated with multiplication by \(-1\).
Hence \Cref{intro-bns-iff-fg} states that \(N\) is finitely generated if and only if the preimages of the positive and of the negative numbers under the projection map are connected as subsets of the Cayley graph.
As an example have a look at \Cref{fig-cay-fg}.
We see the case where \(G = \Z^2\) and \(N\) is the subgroup generated by one of the free generators.
Here, \(N\) is finitely generated and both the positive and negative half of the Cayley graph are connected.
On the other hand, if \(G\) is the free group on two generators, then the kernel of the projection onto one of the generators is not finitely generated and each of the two halves is disconnected.

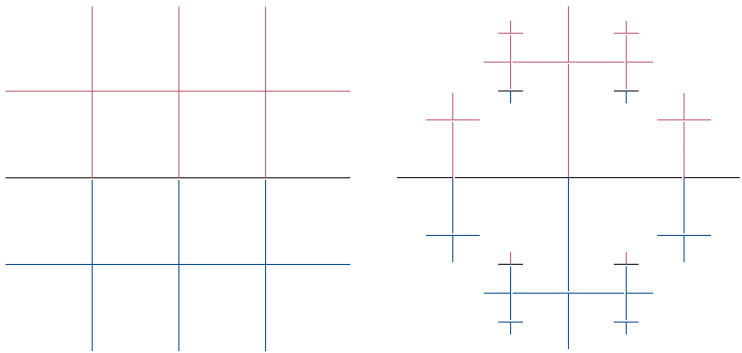
\begin{figure}
	\centering
	\begin{subfigure}{.4\textwidth}
		\resizebox{\textwidth}{!}{
			\begin{tikzpicture}[shorten >=1pt,on grid,auto]

	\node[] (00) {};
	\node[] (01) [right=6 of 00] {};
	\node[] (10) [above=1.5 of 00] {};
	\node[] (11) [right=6 of 10] {};
	\node[] (20) [above=1.5 of 10] {};
	\node[] (21) [right=6 of 20] {};
%	\node[] (30) [above=1 of 20] {};
%	\node[] (31) [right=6 of 30] {};
%	\node[] (40) [above=1 of 30] {};
%	\node[] (41) [right=6 of 40] {};

	\node[] (v00) [below right=1.5 and 1.5 of 00] {};
	\node[] (v01) [right=1.5 of 10] {};
	\node[] (v02) [above=6 of v00] {};
	\node[] (v10) [right=1.5 of v00] {};
	\node[] (v11) [right=1.5 of v01] {};
	\node[] (v12) [right=1.5 of v02] {};
	\node[] (v20) [right=1.5 of v10] {};
	\node[] (v21) [right=1.5 of v11] {};
	\node[] (v22) [right=1.5 of v12] {};
%	\node[] (v30) [right=1 of v20] {};
%	\node[] (v31) [right=1 of v21] {};
%	\node[] (v32) [right=1 of v22] {};
%	\node[] (v40) [right=1 of v30] {};
%	\node[] (v41) [right=1 of v31] {};
%	\node[] (v42) [right=1 of v32] {};

%	\node[] (22) [right=4 of 21] {};
%	\node[] (tr) [above right=1.5 and 2 of 21] {\(\inv\Phi([0,\infty))\)};
%	\node[] (br) [below right=1.5 and 2 of 21] {\(\inv\Phi((-\infty,0])\)};

	\path[-]
		(00.center) edge [blue] (01.center)
		(10.center) edge (11.center)
		(20.center) edge [red] (21.center)
%		(30.center) edge [red] (31.center)
%		(40.center) edge [red] (41.center)

		(v00.center) edge [blue] (v01.center)
		(v01.center) edge [red] (v02.center)
		(v10.center) edge [blue] (v11.center)
		(v11.center) edge [red] (v12.center)
		(v20.center) edge [blue] (v21.center)
		(v21.center) edge [red] (v22.center)
%		(v30.center) edge [blue] (v31.center)
%		(v31.center) edge [red] (v32.center)
%		(v40.center) edge [blue] (v41.center)
%		(v41.center) edge [red] (v42.center)

%		(21.center) edge [dashed, thick] (22.center)
	;
\end{tikzpicture}
		}
	\end{subfigure}
	\begin{subfigure}{.4\textwidth}
		\resizebox{\textwidth}{!}{
			\begin{tikzpicture}[scale=.8, shorten >=1pt,on grid,auto]

	\node[] (1) {};
	\node[] (a) [right=2 of 1] {};
	\node[] (A) [left=2 of 1] {};
	\node[] (b) [above=2 of 1] {};
	\node[] (B) [below=2 of 1] {};

	\node[] (aa) [right=1 of a] {};
	\node[] (ab) [above=1 of a] {};
	\node[] (aB) [below=1 of a] {};

	\node[] (aba) [right=.5 of ab] {};
	\node[] (abb) [above=.5 of ab] {};
	\node[] (abA) [left=.5 of ab] {};

	\node[] (aBa) [right=.5 of aB] {};
	\node[] (aBB) [below=.5 of aB] {};
	\node[] (aBA) [left=.5 of aB] {};

	\node[] (AA) [left=1 of A] {};
	\node[] (Ab) [above=1 of A] {};
	\node[] (AB) [below=1 of A] {};

	\node[] (Aba) [right=.5 of Ab] {};
	\node[] (Abb) [above=.5 of Ab] {};
	\node[] (AbA) [left=.5 of Ab] {};

	\node[] (ABa) [right=.5 of AB] {};
	\node[] (ABB) [below=.5 of AB] {};
	\node[] (ABA) [left=.5 of AB] {};

	\node[] (ba) [right=1 of b] {};
	\node[] (bb) [above=1 of b] {};
	\node[] (bA) [left=1 of b] {};

	\node[] (baa) [right=.5 of ba] {};
	\node[] (bab) [above=.5 of ba] {};
	\node[] (baB) [below=.5 of ba] {};

	\node[] (baba) [right=.25 of bab] {};
	\node[] (babb) [above=.25 of bab] {};
	\node[] (babA) [left=.25 of bab] {};

	\node[] (baBa) [right=.25 of baB] {};
	\node[] (baBB) [below=.25 of baB] {};
	\node[] (baBA) [left=.25 of baB] {};

	\node[] (bAA) [left=.5 of bA] {};
	\node[] (bAb) [above=.5 of bA] {};
	\node[] (bAB) [below=.5 of bA] {};

	\node[] (bAba) [right=.25 of bAb] {};
	\node[] (bAbb) [above=.25 of bAb] {};
	\node[] (bAbA) [left=.25 of bAb] {};

	\node[] (bABa) [right=.25 of bAB] {};
	\node[] (bABB) [below=.25 of bAB] {};
	\node[] (bABA) [left=.25 of bAB] {};

	\node[] (Ba) [right=1 of B] {};
	\node[] (BB) [below=1 of B] {};
	\node[] (BA) [left=1 of B] {};

	\node[] (Baa) [right=.5 of Ba] {};
	\node[] (Bab) [above=.5 of Ba] {};
	\node[] (BaB) [below=.5 of Ba] {};

	\node[] (Baba) [right=.25 of Bab] {};
	\node[] (Babb) [above=.25 of Bab] {};
	\node[] (BabA) [left=.25 of Bab] {};

	\node[] (BaBa) [right=.25 of BaB] {};
	\node[] (BaBB) [below=.25 of BaB] {};
	\node[] (BaBA) [left=.25 of BaB] {};

	\node[] (BAA) [left=.5 of BA] {};
	\node[] (BAb) [above=.5 of BA] {};
	\node[] (BAB) [below=.5 of BA] {};

	\node[] (BAba) [right=.25 of BAb] {};
	\node[] (BAbb) [above=.25 of BAb] {};
	\node[] (BAbA) [left=.25 of BAb] {};

	\node[] (BABa) [right=.25 of BAB] {};
	\node[] (BABB) [below=.25 of BAB] {};
	\node[] (BABA) [left=.25 of BAB] {};

	\path[-]
		(1.center) edge (a.center)
		(1.center) edge (A.center)
		(1.center) edge [red] (b.center)
		(1.center) edge [blue] (B.center)

		(a.center) edge (aa.center)
		(a.center) edge [red] (ab.center)
		(a.center) edge [blue] (aB.center)

		(ab.center) edge [red] (abA.center)
		(ab.center) edge [red] (abb.center)
		(ab.center) edge [red] (aba.center)

		(aB.center) edge [blue] (aBA.center)
		(aB.center) edge [blue] (aBB.center)
		(aB.center) edge [blue] (aBa.center)

		(A.center) edge (AA.center)
		(A.center) edge [red] (Ab.center)
		(A.center) edge [blue] (AB.center)

		(Ab.center) edge [red] (AbA.center)
		(Ab.center) edge [red] (Abb.center)
		(Ab.center) edge [red] (Aba.center)

		(AB.center) edge [blue] (ABA.center)
		(AB.center) edge [blue] (ABB.center)
		(AB.center) edge [blue] (ABa.center)

		(b.center) edge [red] (ba.center)
		(b.center) edge [red] (bb.center)
		(b.center) edge [red] (bA.center)

		(ba.center) edge [red] (bab.center)
		(ba.center) edge [red] (baa.center)
		(ba.center) edge [red] (baB.center)

		(bab.center) edge [red] (baba.center)
		(bab.center) edge [red] (babA.center)
		(bab.center) edge [red] (babb.center)

		(baB.center) edge (baBa.center)
		(baB.center) edge (baBA.center)
		(baB.center) edge [blue] (baBB.center)

		(bA.center) edge [red] (bAb.center)
		(bA.center) edge [red] (bAA.center)
		(bA.center) edge [red] (bAB.center)

		(bAb.center) edge [red] (bAba.center)
		(bAb.center) edge [red] (bAbA.center)
		(bAb.center) edge [red] (bAbb.center)

		(bAB.center) edge (bABa.center)
		(bAB.center) edge (bABA.center)
		(bAB.center) edge [blue] (bABB.center)

		(B.center) edge [blue] (Ba.center)
		(B.center) edge [blue] (BB.center)
		(B.center) edge [blue] (BA.center)

		(Ba.center) edge [blue] (Bab.center)
		(Ba.center) edge [blue] (Baa.center)
		(Ba.center) edge [blue] (BaB.center)

		(Bab.center) edge (Baba.center)
		(Bab.center) edge (BabA.center)
		(Bab.center) edge [red] (Babb.center)

		(BaB.center) edge [blue] (BaBa.center)
		(BaB.center) edge [blue] (BaBA.center)
		(BaB.center) edge [blue] (BaBB.center)

		(BA.center) edge [blue] (BAb.center)
		(BA.center) edge [blue] (BAA.center)
		(BA.center) edge [blue] (BAB.center)

		(BAb.center) edge (BAba.center)
		(BAb.center) edge (BAbA.center)
		(BAb.center) edge [red] (BAbb.center)

		(BAB.center) edge [blue] (BABa.center)
		(BAB.center) edge [blue] (BABA.center)
		(BAB.center) edge [blue] (BABB.center)
	;
\end{tikzpicture}
		}
	\end{subfigure}
	\caption{Cayley graphs of \(\Z^2\) and \(F_2\) divided into positive and negative elements by the projections onto one generator, distinguished by color.}
	\label{fig-cay-fg}
\end{figure}
\subsection*{Structural notes}

Our first objective in this article is to translate characters on \(G\) into orders on \(G\) as outlined above.
This is covered by \textbf{\Cref{partially-ordered-groups}}.

In \textbf{\Cref{bns}}, we review the \(\Sigma\)-invariant and reformulate its definition in terms of orders.
We then study what orders on nilpotent groups look like in \textbf{\Cref{nilpotent-orders}} and prove \Cref{intro-classification-nilpotent-orders}.
We conclude by proving \Cref{intro-bns-iff-fg-nilpotent} in \textbf{\Cref{nilpotent}}.

\subsection*{Acknowledgements}
I would like to thank Julia Heller, Claudio Llosa Isenrich and Roman Sauer for numerous helpful conversations.
I am most grateful to Dawid Kielak for kindly hosting me in Oxford during the early stages of this work and for the many hours we spent discussing it.
While working on this article, I was supported by project 281869850 funded by the Deutsche Forschungsgemeinschaft (DFG, German research foundation).

\section{Partially ordered groups}\label{partially-ordered-groups}

We start by having a look at partially ordered groups in general.
A more thorough introduction may be found for example in \cite{Kopytov}.
It will be of special interest how we can translate between orders and maps.
As far as I am aware, there is no source where this relation is made explicit.
Some of it may be found in sources like \cite{Kopytov} or \cite{Glass}.

\begin{Definition}
	Let \(G\) be a group.
	A \emph{partial order} on \(G\) is a relation \(\ordl\) on the set \(G\) such that for all \(f,g,h \in G\)
	\begin{itemize}
		\item \(g \ordnl g\) \hfill (antireflexive)
		\item \(g \ordl h \Rightarrow h \ordnl g\) \hfill (antisymmetric)
		\item \(f \ordl g \land g \ordl h \Rightarrow f \ordl h\) \hfill (transitive)
	\end{itemize}
	Additionally, a partial order may have the following properties
	\begin{itemize}
		\item \(f \ordl g \Rightarrow hf \ordl hg\) \hfill (left invariant)
		\item \(f \ordl g \Rightarrow fh \ordl gh\) \hfill (right invariant)
	\end{itemize}
	and \(\ordl\) is said to be \emph{bi-invariant} if it satisfies both.
\end{Definition}

\begin{Remark}\begin{enumerate}\item[]
	\item
		The symbol \(=\) always refers to honest equality as elements of \(G\).

	\item
		For an order \(\ordl\), we set \[
			g \ordle h \iff g \ordl h \lor g = h.
		\]
		The symbols \(\ordg\) and \(\ordge\) are defined as the respective opposite orders.

	\item
		In the literature one often finds the definition of \((G, \ordle)\) as definition of a partially ordered group and the symbols \(\ordl, \ordge, \ordg\) are then derived.
		But I find it often makes notation easier to view \(\ordl\) as ``the order'' so that is the convention we use here.

	\item
		When we use \(<, \leqslant, >\) and \(\geqslant\), we always mean them in their well established, canonical meanings, such as the standard order on the real numbers.
\end{enumerate}\end{Remark}

Let us have a look at some important examples.
We will review them in more detail in \Cref{induced-orders} and \Cref{lexicographic-order}.

\begin{Example}\label{initial-example-po}\begin{enumerate}\item[]
	\item
		The \emph{trivial order} \(g \ordnl h\) for each pair \(g, h \in G\) is a partial order for any group \(G\).
	\item
		For any group homomorphism \(\Phi\colon G \to \R\), we can define a bi-invariant partial order \(\ordl\) on \(G\) by letting \[
			g \ordl h \iff \Phi(g) < \Phi(h).
		\]
		For \(\Phi = 0\), we obtain the trivial order on \(G\).
	\item
		For \(G = \Z^2 = \grp{a,b}{[a,b]}\), there is the \emph{lexicographic order} \[a^ib^j \ordl a^kb^l \iff i \leq k \land i=k \Rightarrow j < l.\]
\end{enumerate}\end{Example}

We say \((G, \ordl)\) or sometimes just \(G\) is an \emph{ordered group} as a shortcut for \(G\) being a group and \(\ordl\) a bi-invariant partial order on \(G\).
In this article, any order will be partial and bi-invariant unless noted otherwise.

\begin{Definition}
	Let \(G\) be an ordered group.
	\begin{enumerate}
		\item Two elements \(g, h \in G\) are called \emph{comparable} if \(g \ordl h \lor h \ordl g\) and \emph{incomparable} otherwise.
		\item An order is called \emph{total}, if any two elements are either comparable or equal.
		\item \(g \in G\) is called \emph{positive} if \(1 \ordl g\) and negative if \(g \ordl 1\).
		\item \(\poscone G \coloneqq \{g \in G \mid g \ordg 1\}\) is called the \emph{positive cone} of \((G, \ordl)\).
		\item For the inclusion of a subset \(\iota\colon S \into G\) the \emph{restriction} is the partial order \(\restr\ordl S\) on \(S\) such that for any \(s, t \in S\) \[s \restr\ordl S t \iff \iota(s) \ordl \iota(t).\]
		\item A subset \(S \subseteq G\) is called an \emph{antichain} if \(\restr\ordl S\) is the trivial order.
		\item
			An antichain \(S\) is \emph{maximal} if the only antichain containing \(S\) is \(S\) itself.
			We say that \(S\) is a \emph{maximal antichain subgroup} if it is an antichain and a subgroup and it is maximal amongst subgroups that are antichains.
		\item
			For a subset \(S \subseteq G\) and \(*\) any of \(\{\ordl, \ordg, \ordle, \ordge\}\), set \[
				S_* \coloneqq \{g \in G \mid \exists s \in S\colon g * s\}.
			\]
	\end{enumerate}
\end{Definition}

\begin{Remark}\label{torsion-is-antichain}
	\begin{enumerate}\item[]
		\item Being incomparable is in general not transitive.
		\item
			A subgroup \(H \subg G\) is an antichain if and only if \(H \cap \poscone G\) is the empty set.
			The same cannot be said if \(H\) is just any subset.
		\item
			For any subset \(S \subseteq G\) we have \(S_{\ordl} \cap S_{\ordge} = \emptyset\) if and only if \(S\) is an antichain.
			Additionally, \(S_{\ordl} \cup S_{\ordge} = G\) if and only if \(S\) is a maximal antichain.
			The latter is not always true if \(S\) is a maximal antichain subgroup as we will see for instance in \Cref{example-abelian-orders}.
		\item
			If \(g\) is a torsion element of \(G\), then \(1\) and \(g\) are necessarily incomparable since if \(g\) is positive, then \[
				1 \ordl g \ordl g^2 \ordl \dots \ordl g^n = 1
			\] contradicts antireflexivity and analogously if \(g\) is negative.

			If the set of all torsion elements \(T\) is a normal subgroup, then any order on \(G\) is induced by an order on \(G/T\) - a notion we will make precise in \Cref{induced-orders}.
			Since this will almost always be the case in this work, it suffices for us to think of \(G\) as torsion free.
	\end{enumerate}
\end{Remark}

It is often useful and allows for more ergonomic notation to think of the positive cone instead of the order itself.
The following lemma tells us that an order is uniquely determined by its positive cone.

\begin{Lemma}\label{order-from-positive}
	Let \(G\) be a group.
	\begin{enumerate}
		\item	If \(G\) is ordered, the positive cone \(\poscone G\) is closed under multiplication and under conjugation with elements of \(G\) and if \(g \in \poscone G\) then \(\inv g \notin \poscone G\).
		\item
			For any subset \(S \subseteq G\) that is closed under multiplication and under conjugation with elements of \(G\) and such that \(S \cap \inv S = \emptyset\),
			there is a unique order on \(G\) such that \(\poscone G = S\).
	\end{enumerate}
\end{Lemma}
\begin{proof}\begin{enumerate}\item[]
\item
	This is immediate from the definition of a bi-invariant partial order.

\item
	The postulated order is \[
		g \ordl h \iff \inv gh \in S.
	\]
	Checking that this is indeed an order is again essentially just applying the definition.
	Transitivity follows from the assumption that \(S\) is closed under multiplication.
	For anti-symmetry, we use that \(S \cap \inv S = \emptyset\).
	In particular, \(1 \notin S\), so \(\ordl\) is anti-reflexive.
	Left invariance is straightforward and for right invariance we need \(S\) closed under conjugation.

	As for uniqueness, note that \[g \ordl h \iff 1 \ordl \inv gh \iff \inv gh \in \poscone G = S\] is necessarily true for any bi-invariant relation with \(\poscone G = S\).
\end{enumerate}\end{proof}
Our goal later on will be to translate properties of group homomorphisms into a language based on partial orders.
To see how the two concepts relate, we introduce the following notions.

\begin{Definition}\label{induced-orders}
	Let \((G, \ordl)\) be an ordered group.
	\begin{enumerate}
		\item
			Let \(Q\) be another ordered group and \(\varphi\colon G \to Q\) a group homomorphism.
			Then \(\varphi\) is called \emph{order-preserving} if for all \(g, h \in G\) we have \[
				g \ordl h \Rightarrow \varphi(g) \ordle \varphi(h).
			\]

		\item
			Let \(\ordl'\) be another order on \(G\).
			We say that \(\ordl\) is a \emph{suborder} of \(\ordl'\) if \[\poscone G \subseteq \posconex G{\ordg'}.\]

		\item
			Let \((Q, \ordl_Q)\) be an ordered group and \(\Phi\colon G \to Q\) order preserving.
			We say that \(\ordl\) \emph{is induced} by \((\Phi, \ordl_Q)\) if the following condition holds:
			For every order \(\ordl'\) on \(G\) such that \(\Phi\colon (G, \ordl') \to (Q, \ordl_Q)\) is order preserving, \(\ordl'\) is a suborder of \(\ordl\).

			We also say that \(\ordl\) is induced by \(\Phi\) or by \(\ordl_Q\) if the other object is clear from the context.

		\item
			Let \((H, \ordl_H)\) be an ordered group and \(\iota \colon H \to G\) order preserving.
			We say that \(\ordl\) \emph{is induced} by \((\iota, \ordl_H)\) if the following condition holds:
			For every order \(\ordl'\) on \(G\) such that \(\iota\colon (H, \ordl_H) \to (G, \ordl')\) is order preserving, \(\ordl\) is a suborder of \(\ordl'\).

			Again, we say that \(\ordl\) is induced by \(\iota\) or by \(\ordl_H\) if the other is clear.
			We denote the induced orders by \(\indor{\ordl_Q}\Phi\) and \(\indol{\ordl_H}\iota\) respectively or just by \(\ordl\) if there is no chance of confusion.

		\item
			If \((Q, \ordl_Q)\) is an ordered group and \(\Phi\colon G \to Q\) such that \(\ordl\) is induced by \((\Phi, \ordl_Q)\), then \(\Phi\) is \emph{order inducing on the domain}.
			An \emph{order inducing map on the codomain} \(\iota \colon H \to G\) is defined analogously.
			We will omit the (co)domain part if it is clear on which side a map is order inducing.
	\end{enumerate}
\end{Definition}

\begin{Remark}
	The map \(\iota\) is not necessarily injective.
	However, if \(\iota\) induces an order on its image, then we may also describe that order as the order induced by the inclusion \(\iota\colon H/\ker\iota \into G\).
	The projection \(H \onto H/\ker\iota\) is necessarily order preserving in this case.
	Thus \(\iota\) can be made injective without losing any information whenever \(\iota\) is order inducing.
\end{Remark}

In many cases, if \(\ordl\) is an induced order, it is possible to make its positive cone explicit.
Most of the time I find it easier to think about induced orders using the following description.

\begin{Lemma}\label{explicit-induced-order}
	Let \(G\) be a group and \(Q, H\) be ordered groups.
	Let \(\Phi\colon G \to Q\) and \(\iota\colon H \to G\).

	Then \begin{enumerate}
		\item\label{explicit-induced-order-1}
			\(\Phi\) induces the following order on \(G\): \[
				1 \indor\ordl\Phi g \iff 1 \ordl \Phi(g)
			\]
		\item\label{explicit-induced-order-2}
			Consider the relation \[
				1 \ordl' g \iff g \in G \cdot \iota(\poscone H)
			\]
			Here, \(G \cdot \iota(\poscone H)\) denotes the image of \(\iota(\poscone H)\) under conjugation with elements in \(G\).
			If \(\iota\) induces an order on \(G\), then \(\ordl'\) is an order and it is the order induced by \(\iota\).
	\end{enumerate}
\end{Lemma}
\begin{proof} \begin{enumerate}\item[]\item
	This is the order such that \(\poscone G = \inv\Phi(\poscone Q)\).
	Checking that \(\inv\Phi(\poscone Q)\) is closed under multiplication, conjugation with \(g \in G\) and does not contain two elements inverse to each other is straightforward since we know that \(\poscone Q\) has all these properties.
	The order is then well-defined by \Cref{order-from-positive}.

	As any order on \(G\) such that \(\Phi\) is order preserving requires \(\Phi(\poscone G) \subseteq \poscone Q\), all such orders are suborders of \(\indor\ordl\Phi\).

\item
	If \(\iota\) is order preserving, then every element of \(\iota(\poscone H)\) has to be positive.
	Because the positive cone of the induced order has to be closed under conjugation, this extends to \(G \cdot \iota(\poscone H)\).
	Hence if the relation \(\ordl'\) defined in the statement is actually an order, then it is a suborder of every order such that \(\iota\) is order preserving.

	Checking that \(G \cdot \iota(\poscone H)\) is closed under multiplication and conjugation is straightforward.
	It may however contain two elements inverse to each other.
	But in this case, the above argument shows that \(\iota\) cannot be order preserving for any order on \(G\) and hence \(\iota\) does not induce any order.
\end{enumerate}\end{proof}

\begin{Remark}
	In particular, any map \(\Phi\) with an ordered codomain induces an order on the domain.
	But not every map \(\iota\) with an ordered domain induces an order on the codomain.
	However, it is easy to tell if a given order on the codomain is induced by \(\iota\).
	Especially if \(\iota(\poscone H)\) is already closed under conjugation with \(G\), then this is the positive cone of the order induced by \(\iota\).
	Also, if \(G\) admits some order such that \(\poscone G \subseteq \iota(H)\), then that order is induced by the inclusion of \(H\).
\end{Remark}

We may chain order-inducing maps as expected:
\begin{Lemma}\label{inducing-maps-chain}
	Let there be three ordered groups \((G, \ordl_G),\enskip (H, \ordl_H)\) and \((K, \ordl_K)\). 
	Let \mbox{\(\Phi\colon G \to H\)} and \mbox{\(\Psi \colon H \to K\)} be maps of groups.
	\begin{enumerate}
		\item If \(\ordl_H\) is induced by \(\Psi\) and \(\ordl_G\) is induced by \(\Phi\), then \(\ordl_G\) is also induced by \(\Psi \circ \Phi\).
		\item If \(\ordl_H\) is induced by \(\Phi\) and \(\ordl_K\) is induced by \(\Psi\), then \(\ordl_K\) is also induced by \(\Psi \circ \Phi\).
	\end{enumerate}
\end{Lemma}
\begin{proof}
	For the first part, from \Cref{explicit-induced-order} we know that \[
		1 \ordl_G g \iff 1 \ordl_H \Phi(G) \iff 1 \ordl_K (\Psi \circ \Phi)(g)
	\] and hence \(\ordl_G = \indor{\ordl_K}{(\Psi \circ \Phi)}\).
	The second part may be proven similarly.
\end{proof}

The kernel of a map will be of particular interest to us later on.
Given a map that is order inducing on the domain, we would like to know if we can recover the kernel of that map just by looking at the induced order.
This is not always possible.
For example in the case where both the inducing and the induced order are trivial.
But in most other cases, we get some restrictions on what the kernel might have been.

\begin{Lemma}\label{ker-max-ac}
	Let \(G\) be a finitely generated group,
	\(Q\) an ordered group,
	\(\Phi\colon G \onto Q\) onto
	and let \(G\) carry the order induced by \(\Phi\).
	Then: \begin{enumerate}
		\item\label{ker-max-ac-1} If \(K \subseteq G\) is an antichain, then so is \(\Phi(K)\). If \(K\) is maximal then \(\Phi(K)\) is also maximal.
		\item\label{ker-max-ac-2} If \(P \subseteq Q\) is an antichain, then so is \(\inv\Phi(P)\). If \(P\) is maximal then \(\inv\Phi(P)\) is also maximal.
		\item For any maximal antichain \(K \subseteq G\) containing \(1\) we have \(\ker\Phi \subseteq K\).
		\item If \(Q\) is totally ordered, \(\ker\Phi\) is an antichain and every other antichain that contains \(1\) is contained in \(\ker\Phi\).
			In particular, \(\ker\Phi\) is the only maximal antichain subgroup.
			It is maximal even among all antichains.
	\end{enumerate}
\end{Lemma}
\begin{proof}\begin{enumerate}\item[]
\item
	As \(\Phi\) is order preserving, \(\Phi(k) \ordl \Phi(k')\) implies \(k \ordl k'\), so \(\Phi(K)\) is an antichain.

	Now suppose there is a \(q \in Q \setminus \Phi(Q)\) such that \(q\) is incomparable to every \(k \in \Phi(K)\).
	Then any preimage \(q_0\) of \(q\) is incomparable to any element of \(K\).
	This contradicts the maximality of \(K\), so \(\Phi(K)\) is also maximal.

\item
	Suppose there are comparable \(g, h \in \inv\Phi(P)\).
	Then \(\Phi(g), \Phi(h) \in P\) are also comparable but \(P\) is an antichain.
	Hence \(\inv\Phi(P)\) must be an antichain.

	Suppose there is a larger antichain \(\inv\Phi(P) \subsetneq H \subseteq G\).
	Then \(P \subsetneq \Phi(H)\) also is an antichain by \ref{ker-max-ac-1}.
	So \(P\) cannot be maximal.

\item
	By \ref{ker-max-ac-1}, \(\Phi(K)\) is a maximal antichain in \(Q\) that contains \(1\).
	\ref{ker-max-ac-2} tells us that then \(K = \inv\Phi(\Phi(K)) \supseteq \inv\Phi(1) = \ker\Phi\).

\item
	If \(Q\) is totally ordered, the only maximal antichain containing \(1\) is \(\{1\}\).
	By \ref{ker-max-ac-2}, \(\ker\Phi\) is a maximal antichain and by \ref{ker-max-ac-1} there cannot be any other that contains \(1\).
	Hence \(\ker\Phi\) contains all other antichains that contain \(1\).
\end{enumerate}\end{proof}

Let us have a look at an example class of orders that will show up multiple times in this article.
\begin{Definition}\label{lexicographic-order}
	Let \[0 \to H \overset\iota\into G \overset\pi\onto Q \to 0\] be an exact sequence of groups and
	suppose that \(H\) and \(Q\) are ordered by \(\ordl_H\) and \(\ordl_Q\) respectively.
	Further assume that \(\iota\) induces an order on \(G\).

	Then the \emph{lexicographic order} with respect to this sequence is the order \(\ordl_G\) such that \[
		1 \ordl_G g \iff g \in \inv\pi(\poscone Q) \lor g \in \iota(\poscone H) 
	\]
\end{Definition}

\begin{Remark}
	For a lexicographic order with respect to some sequence \(H \into G \onto Q\), note that any two given elements either have different images in \(Q\) or they lie in the same \(H\)-coset.
	Hence in order to compare those elements, we first try comparing their images in \(Q\).
	If they have the same image, we instead compare them using the order induced by \(H\).
	In case \(G\) is a semidirect product \(G = H \ltimes Q\), this means comparing the \(Q\) factor and then the \(H\) factor.

	The lexicographic order has both \(\indor{\ordl_Q}\pi\) and \(\indol{\ordl_H}\iota\) as suborders.
	In fact, it is the smallest such order.
	So one might think about it as ``induced by \(\pi\) and \(\iota\) together''.
\end{Remark}

\begin{figure}
	\centering
	\begin{subfigure}{.45\textwidth}
		\centering
		\begin{tikzpicture}[shorten >=1pt,on grid,auto]

	\node[] (00) {};
	\node[] (01) [right=4 of 00] {};
	\node[] (10) [above=1 of 00] {};
	\node[] (11) [right=4 of 10] {};
	\node[] (20) [above=1 of 10] {};
	\node[] (21) [right=4 of 20] {};

	\node[] (v00) [below right=1 and 1 of 00] {};
	\node[] (v01) [right=1 of 10] {};
	\node[] (v02) [above=4 of v00] {};
	\node[] (v10) [right=1 of v00] {};
	\node[] (v11) [right=1 of v01] {};
	\node[] (v12) [right=1 of v02] {};
	\node[] (v20) [right=1 of v10] {};
	\node[] (v21) [right=1 of v11] {};
	\node[] (v22) [right=1 of v12] {};
	\path[-]
		(00.center) edge [gray] (01.center)
		(10.center) edge [gray] (11.center)
		(20.center) edge [gray] (21.center)
		(v00.center) edge [gray] (v01.center)
		(v01.center) edge [gray] (v02.center)
		(v10.center) edge [gray] (v11.center)
		(v11.center) edge [gray] (v12.center)
		(v20.center) edge [gray] (v21.center)
		(v21.center) edge [gray] (v22.center)
	;

	\node[] (l0) [left=.4 of v00] {};
	\node[] (l1) [right=.4 of v22] {};

	\path[-]
		(l0.center) edge (l1.center)
	;

	\node[] (p00) [right=1 of 00] {};
	\node[] (p01) [right=1 of p00] {};
	\node[] (p02) [right=1 of p01] {};
	\node[] (p10) [above=1 of p00] {};
	\node[] (p12) [right=2 of p10] {};
	\node[] (p20) [above=1 of p10] {};
	\node[] (p21) [right=1 of p20] {};
	\node[] (p22) [right=1 of p21] {};

	\draw[red]  (p00.center) -- ($(l0.center)!(p00.center)!(l1.center)$);
	\draw[blue] (p01.center) -- ($(l0.center)!(p01.center)!(l1.center)$);
	\draw[blue] (p02.center) -- ($(l0.center)!(p02.center)!(l1.center)$);
	\draw[red]  (p10.center) -- ($(l0.center)!(p10.center)!(l1.center)$);
	\draw[blue] (p12.center) -- ($(l0.center)!(p12.center)!(l1.center)$);
	\draw[red]  (p20.center) -- ($(l0.center)!(p20.center)!(l1.center)$);
	\draw[red]  (p21.center) -- ($(l0.center)!(p21.center)!(l1.center)$);
	\draw[blue] (p22.center) -- ($(l0.center)!(p22.center)!(l1.center)$);
\end{tikzpicture}
		\caption{total, full, archimedean}
		\label{ord-z2-total}
	\end{subfigure}
	\begin{subfigure}{.45\textwidth}
		\centering
		\begin{tikzpicture}[shorten >=1pt,on grid,auto]

	\node[] (00) {};
	\node[] (01) [right=4 of 00] {};
	\node[] (10) [above=1 of 00] {};
	\node[] (11) [right=4 of 10] {};
	\node[] (20) [above=1 of 10] {};
	\node[] (21) [right=4 of 20] {};

	\node[] (v00) [below right=1 and 1 of 00] {};
	\node[] (v01) [right=1 of 10] {};
	\node[] (v02) [above=4 of v00] {};
	\node[] (v10) [right=1 of v00] {};
	\node[] (v11) [right=1 of v01] {};
	\node[] (v12) [right=1 of v02] {};
	\node[] (v20) [right=1 of v10] {};
	\node[] (v21) [right=1 of v11] {};
	\node[] (v22) [right=1 of v12] {};
	\path[-]
		(00.center) edge [gray] (01.center)
		(10.center) edge [gray] (11.center)
		(20.center) edge [gray] (21.center)
		(v00.center) edge [gray] (v01.center)
		(v01.center) edge [gray] (v02.center)
		(v10.center) edge [gray] (v11.center)
		(v11.center) edge [gray] (v12.center)
		(v20.center) edge [gray] (v21.center)
		(v21.center) edge [gray] (v22.center)
	;

	\node[] (l0) [left=1 of v00] {};
	\node[] (l1) [right=1 of v22] {};

	\path[-]
		(l0.center) edge (l1.center)
	;

	\node[] (p00) [right=1 of 00] {};
	\node[] (p01) [right=1 of p00] {};
	\node[] (p02) [right=1 of p01] {};
	\node[] (p10) [above=1 of p00] {};
	\node[] (p12) [right=2 of p10] {};
	\node[] (p20) [above=1 of p10] {};
	\node[] (p21) [right=1 of p20] {};
	\node[] (p22) [right=1 of p21] {};

	\draw[blue] (p01.center) -- ($(l0.center)!(p01.center)!(l1.center)$);
	\draw[blue] (p02.center) -- ($(l0.center)!(p02.center)!(l1.center)$);
	\draw[red]  (p10.center) -- ($(l0.center)!(p10.center)!(l1.center)$);
	\draw[blue] (p12.center) -- ($(l0.center)!(p12.center)!(l1.center)$);
	\draw[red]  (p20.center) -- ($(l0.center)!(p20.center)!(l1.center)$);
	\draw[red]  (p21.center) -- ($(l0.center)!(p21.center)!(l1.center)$);
\end{tikzpicture}
		\caption{non-total, full, archimedean}
		\label{ord-z2-onto-z}
	\end{subfigure}
	\begin{subfigure}{.45\textwidth}
		\centering
		\begin{tikzpicture}[shorten >=1pt,on grid,auto]

	\node[] (00) {};
	\node[] (01) [right=4 of 00] {};
	\node[] (10) [above=1 of 00] {};
	\node[] (11) [right=4 of 10] {};
	\node[] (20) [above=1 of 10] {};
	\node[] (21) [right=4 of 20] {};

	\node[] (v00) [below right=1 and 1 of 00] {};
	\node[] (v01) [right=1 of 10] {};
	\node[] (v02) [above=4 of v00] {};
	\node[] (v10) [right=1 of v00] {};
	\node[] (v11) [right=1 of v01] {};
	\node[] (v12) [right=1 of v02] {};
	\node[] (v20) [right=1 of v10] {};
	\node[] (v21) [right=1 of v11] {};
	\node[] (v22) [right=1 of v12] {};
	\path[-]
		(00.center) edge [gray] (01.center)
		(10.center) edge [gray] (11.center)
		(20.center) edge [gray] (21.center)
		(v00.center) edge [gray] (v01.center)
		(v01.center) edge [gray] (v02.center)
		(v10.center) edge [gray] (v11.center)
		(v11.center) edge [gray] (v12.center)
		(v20.center) edge [gray] (v21.center)
		(v21.center) edge [gray] (v22.center)
	;

	\node[] (l0) [left=1 of v00] {};
	\node[] (l1) [right=1 of v22] {};

	\node[] (p00) [right=1 of 00] {};
	\node[] (p01) [right=1 of p00] {};
	\node[] (p02) [right=1 of p01] {};
	\node[] (p10) [above=1 of p00] {};
	\node[] (p11) [right=1 of p10] {};
	\node[] (p12) [right=1 of p11] {};
	\node[] (p20) [above=1 of p10] {};
	\node[] (p21) [right=1 of p20] {};
	\node[] (p22) [right=1 of p21] {};

	\path[-]
		(p00.center) edge [blue] (p11.center)
		(l0.center)  edge [blue] (p00.center)
		(p11.center) edge [red]  (p22.center)
		(p22.center) edge [red]  (l1.center)
	;
\end{tikzpicture}
		\caption{non-total, non-full, archimedean}
		\label{ord-z2-from-z}
	\end{subfigure}
	\begin{subfigure}{.45\textwidth}
		\centering
		\begin{tikzpicture}[shorten >=1pt,on grid,auto]

	\node[] (00) {};
	\node[] (01) [right=4 of 00] {};
	\node[] (10) [above=1 of 00] {};
	\node[] (11) [right=4 of 10] {};
	\node[] (20) [above=1 of 10] {};
	\node[] (21) [right=4 of 20] {};

	\node[] (v00) [below right=1 and 1 of 00] {};
	\node[] (v01) [right=1 of 10] {};
	\node[] (v02) [above=4 of v00] {};
	\node[] (v10) [right=1 of v00] {};
	\node[] (v11) [right=1 of v01] {};
	\node[] (v12) [right=1 of v02] {};
	\node[] (v20) [right=1 of v10] {};
	\node[] (v21) [right=1 of v11] {};
	\node[] (v22) [right=1 of v12] {};
	\path[-]
		(00.center) edge [gray] (01.center)
		(10.center) edge [gray] (11.center)
		(20.center) edge [gray] (21.center)
		(v00.center) edge [gray] (v01.center)
		(v01.center) edge [gray] (v02.center)
		(v10.center) edge [gray] (v11.center)
		(v11.center) edge [gray] (v12.center)
		(v20.center) edge [gray] (v21.center)
		(v21.center) edge [gray] (v22.center)
	;

	\node[] (l0) [left=1 of v00] {};
	\node[] (l1) [right=1 of v22] {};

	\node[] (p00) [right=1 of 00] {};
	\node[] (p01) [right=1 of p00] {};
	\node[] (p02) [right=1 of p01] {};
	\node[] (p10) [above=1 of p00] {};
	\node[] (p11) [right=1 of p10] {};
	\node[] (p12) [right=1 of p11] {};
	\node[] (p20) [above=1 of p10] {};
	\node[] (p21) [right=1 of p20] {};
	\node[] (p22) [right=1 of p21] {};

	\draw[blue] (p01.center) -- ($(l0.center)!(p01.center)!(l1.center)$);
	\draw[blue] (p02.center) -- ($(l0.center)!(p02.center)!(l1.center)$);
	\draw[red]  (p10.center) -- ($(l0.center)!(p10.center)!(l1.center)$);
	\draw[blue] (p12.center) -- ($(l0.center)!(p12.center)!(l1.center)$);
	\draw[red]  (p20.center) -- ($(l0.center)!(p20.center)!(l1.center)$);
	\draw[red]  (p21.center) -- ($(l0.center)!(p21.center)!(l1.center)$);
	\path[-]
		(p00.center) edge [blue] (p11.center)
		(l0.center)  edge [blue] (p00.center)
		(p11.center) edge [red]  (p22.center)
		(p22.center) edge [red]  (l1.center)
	;
\end{tikzpicture}
		\caption{total, full, non-archimedean}
		\label{ord-z2-lex}
	\end{subfigure}
	\caption{Orders on \(\Z^2\)}
	\label{ord-z2}
\end{figure}

With these tools at hand we can construct a multitude of orders.
In \Cref{ord-z2} we see some orders on \(\Z^2\).

\subref{ord-z2-total} is an order induced by an injective map to \(\R\).
Such a map corresponds to a line at an irrational slope.
Elements are ordered by their oriented distance to that line.
This distance is unique for every element so the order is total.

Similarly for \subref{ord-z2-onto-z}.
But here, the inducing map has image isomorphic to \(\Z\), so the line is at a rational slope.
Points on the line are elements of the kernel of the inducing map.
Distances are no longer unique. For example all points on the line have distance \(0\) from it.
Two elements at the same distance are incomparable.

\subref{ord-z2-from-z} is an order that is induced by the inclusion of \(\Z\) into \(\Z^2\).
All elements comparable to \(1\) lie in some copy of \(\Z\) embedded into \(\Z^2\).
In the picture, this is the colored diagonal line.
Comparability divides \(\Z^2\) into equivalence classes that correspond to parallels of the colored lines.

Finally, \subref{ord-z2-lex} is a lexicographic order corresponding to the sequence \(\Z \into \Z^2 \onto \Z\) where both factors are non-trivially ordered.
Note that this picture may be obtained by overlaying \subref{ord-z2-onto-z} with \subref{ord-z2-from-z}.
This corresponds to the fact that we may use either copy of \(\Z\) in the above sequence to compare elements in the lexicographic order.

In \Cref{example-abelian-orders} we will see that every non-trivial order on \(\Z^2\) falls into exactly one of these four categories.

However, our goal is to translate a statement about characters into a statement about orders.
Note that only \subref{ord-z2-total} and \subref{ord-z2-onto-z} are induced by maps to \(\R\).
Thus we need some way to recognise orders that are induced by characters.
That is, we need to distinguish these orders from \subref{ord-z2-from-z} and \subref{ord-z2-lex}.
To this end, the following definitions turn out to be useful.

\begin{Definition}
	Let \(G\) be an ordered group and \(g, h \in G\).
	We say that \(g\) is \emph{infinitesimal with respect to \(h\)}, if \(g^i \ordl h\) for all \(i \in \Z\).
	In this case, we write \(g \ordll h\).

	An order is called \emph{archimedean} if it does not admit any positive infinitesimal elements.
\end{Definition}

\begin{Remark}\label{total-archimedean-in-r}\begin{enumerate}\item[]
\item
	The way we worded the definition, \(1\) is infinitesimal with respect to any positive element.

\item
	If \(\ordl_G\) is a lexicographic order with respect to some orders \(\ordl_Q\) and \(\ordl_H\), then \(\ordl_G\) is archimedean if and only if at least one of \(\ordl_Q\) and \(\ordl_H\) is trivial and the other one is archimedean.
	In particular, an order induced by another order is archimedean if and only if the inducing order is.

\item
	It is a fact due to Hölder \cite{Holder} that every totally ordered archimedean group is a subgroup of \(\R\).
	An account in English may be found for example in \cite{Kopytov}.
	For partially ordered groups, this is not the case as we will see in \Cref{example-nilpotent-orders}.
\end{enumerate}\end{Remark}

\begin{Definition}
	Let \((G, \ordl)\) be an ordered group. \begin{enumerate}
		\item The order \(\ordl\) is called \emph{primitive}, if for every \(g, h \in G\) and every \(n \in \N\) we have \(g^n \ordl h^n \Rightarrow g \ordl h\).
		\item \(\ordl\) is called \emph{factorizing} if for any antichain normal subgroup \(H \normsub G\), \(\ordl\) is induced by the projection \(G \onto G/H\) for some order on \(G/H\).
		\item \(\ordl\) is called \emph{full} if it has both above properties.
	\end{enumerate}
\end{Definition}

\begin{Example}\label{example-abelian-orders}\begin{enumerate}\item[]
\item
	On \(\Z = \grpg X\), there are exactly two non-trivial full orders:
	The standard order where \(X \ordg 1\) and the opposite order, where we interchange the meanings of ``positive'' and ``negative''.
	That is the unique order where \(X \ordl 1\).
	The non-trivial isomorphism \((\Z, \ordl) \to (\Z, \ordg)\) is order inducing in either direction.

	If we don't require our order to be full, any submonoid that does not contain any two elements inverse to each other defines an order on \(\Z\).
	This includes for example the monoids \(2\N \subseteq \Z\) and \(\{1, X^k \mid k \geq 2\} \subseteq \Z\).

\item
	On \(\Z^n\), any primitive order is total or lexicographic.
	Thus we may construct all primitive orders on a finitely generated abelian group by decomposing it into totally or trivially ordered factors.
	Then we order their product stepwise by the lexicographic order of these factors.

	The resulting order is archimedean if at most one factor was non-trivially ordered.
	It is factorizing and hence full if at every step of constructing a lexicographic order, if the left factor is trivially ordered, then so is the right factor.

\item
	To make the previous example more concrete, consider the case \(n=2\) and recall \Cref{ord-z2}.
	When writing \(\Z^2\) as a product of abelian groups, we can either use just one factor \(\Z^2\) and order that totally.
	In this case we obtain an order like \subref{ord-z2-total}.
	Or we decompose it into two factors \(A \into \Z^2 \onto B\), where both \(A\) and \(B\) are isomorphic to \(\Z\).
	Then we can order \(A\) trivially and \(B\) totally as in \subref{ord-z2-onto-z},
	the other way round as in \subref{ord-z2-from-z},
	or if we order both factors totally we get \subref{ord-z2-lex}.

	Note that an order on \(\Z^2\) is induced by a map to \(\R\) if and only if it is full and archimedean.
	Also note that totality is not a useful criterion in this situation and that is why we consider partial orders in the first place.

	The fact that every order is full or archimedean is due to \(n=2\) being too small.
	\(\Z^3\) admits an order that is primitive but neither full nor archimedean.
	Namely the lexicographic order with respect to \(\Z \into \Z^3 \onto \Z^2\), where \(\Z\) is trivially ordered and \(\Z^2\) carries a total non-archimedean order.
	The same order can be realized as a lexicographic order with respect to \(\Z^2 \into \Z^3 \onto \Z\).
\end{enumerate}\end{Example}

\begin{Remark}
	If \(\ordl_G\) is induced by \(G \onto (Q, \ordl_Q)\) and \(\ordl_Q\) is full, then so is \(\ordl_G\).
	But if \(\ordl_G\) is induced by \((H, \ordl_H) \into G\) and \(\ordl_H\) is full, then \(\ordl_G\) need not be full.
	For example if \(2\Z\) is taken to be ordered by the standard order, the order on \(\Z\) induced by \(2\Z \into \Z\) is not full.
\end{Remark}

To finish the example, let us classify full archimedean orders on finitely generated abelian groups.
\begin{Lemma}\label{characterisation-abelian-orders}
	Let \((G, \ordl)\) be a finitely generated abelian group and \(\ordl\) a full archimedean order.
	Then \begin{enumerate}
		\item \(G\) contains a unique maximal antichain subgroup \(H\),
		\item \(G/H\) is totally ordered,
		\item and \(\ordl\) is induced by a map \(\Phi\colon G \to \R\), where \(\R\) carries the standard order.
	\end{enumerate}
\end{Lemma}
\begin{proof}\begin{enumerate}\item[]
\item
	We may write \(G = F \oplus T\) where \(F\) is a free-abelian group and \(T\) is the torsion part.
	As \(T\) is an antichain by \Cref{torsion-is-antichain}, and \(\ordl\) is full, \(\ordl\) is induced by the projection \(G \onto F\).
	Thus we may assume without loss of generality that \(G = \Z^n\) for some \(n \in \N_0\).

	We proceed by induction on \(n\).
	For \(n=0\), the only subgroup is the trivial group and it is indeed an antichain.

	Otherwise, if \(G\) is totally ordered, the trivial group is the only antichain and hence also maximal.

	If \(G\) is not totally ordered, any element incomparable to \(1\) generates an antichain subgroup \(C \subg G\) with \(C \cong \Z\).
	In this case, \(\ordl\) is induced by the projection onto \(G/C \cong \Z^{n-1} \oplus T'\).
	But the torsion part \(T'\) is again trivially ordered so \(\ordl\) is actually induced by the projection \(G \onto \Z^{n-1}\).

	By induction, \(\Z^{n-1}\) contains a unique maximal antichain subgroup \(H'\). Its preimage under the projection is then the unique maximal antichain subgroup of \(G\) by \Cref{ker-max-ac}.

\item
	Let \(g \in G/H\) be incomparable to \(1\) and let \(g_0\) be a preimage of \(g\) under the projection \(G \onto G/H\).
	Then \(g_0\) is also incomparable to \(1\).
	As \(H\) is the only maximal antichain subgroup, \(g_0 \in H\).
	But then, \(g = 1\) so \(G/H\) is totally ordered.

\item
	In \Cref{total-archimedean-in-r} we saw that every totally ordered abelian group is a subgroup of \(\R\).
	Or in other words, every order on such a group is induced by the inclusion of said group into \(\R\).

	We know that \(\ordl\) is induced by an order on \(G/H\) and \(G/H\) is a totally ordered abelian group.
	By \Cref{inducing-maps-chain}, \(\ordl\) is induced by a map \mbox{\(G \onto G/H \into \R\)}.

\end{enumerate}\end{proof}

\section{\(\Sigma\)-invariants for partial orders}\label{bns}

The \(\Sigma\)-invariant \(\bns G\) is a set of characters \(G \to \R\).
By identifying a character with the order it induces on \(G\), we may interpret \(\bns G\) as a set of orders on \(G\) instead.
With this goal in mind, let us start by having a look at the classical setting as introduced in \cite{BNS}.

\begin{Definition}\label{original-bns}
	Let \(G\) be a group with finite generating set \(S\).
	Denote its \emph{character sphere} by \(S(G) \coloneqq (\Hom(G, \R) \setminus 0) / \sim\),
	where \(\Phi \sim \Psi \iff \Phi = \lambda\Psi\) for some \(\lambda > 0\).

	Then the \emph{(first) \(\Sigma\)-invariant} \(\bns G \subseteq S(G)\) is the subset containing all characters \(\Phi\) such that the full subgraph of \(\Cay{G,S}\) spanned by \(\inv\Phi([0,\infty))\) is connected.
	By the \emph{full subgraph spanned by a subset} \(K \subseteq G\) we mean the graph consisting of all vertices \(g \in K\) and all edges in \(\Cay{G,S}\) for which both endpoints are in \(K\).
\end{Definition}

This is not the original formulation that Bieri, Neumann and Strebel used.
The equivalence of this definition here and the original one may be found for example in \cite{Strebel-notes}.
There one also finds the fact that the definition does not depend on the representative modulo \(\sim\), nor on the choice of finite generating set \(S\).

Along with the \(\Sigma\)-invariant comes the relative character sphere which is defined as follows:
\begin{Definition}\label{original-relative-character-sphere}
	Let \(G\) be a group, \(S(G)\) its character sphere and \(N \normsub G\) a normal subgroup.

	Then the \emph{relative character-sphere} is \[
		S(G, N) \coloneqq \{\Phi \in S(G) \mid N \subg \ker\Phi \}
	\]
\end{Definition}

The two are connected to finite connectedness of \(N\) by the following theorem.

\begin{Theorem}[Bieri, Neumann, Strebel \cite{BNS}] \label{original-bns-iff-fg-abelian}
	Let \(G\) be a finitely generated group and \(N \normsub G\) a normal subgroup such that \(G/N\) is abelian.

	Then \(N\) is finitely generated if and only if \(S(G,N) \subseteq \bns G\).
\end{Theorem}

This theorem gives us a criterion for the kernel of a map being finitely generated only if the codomain of that map is an abelian group.
The criterion then has us consider all maps to \(\R\).
In order to generalise to maps onto non-abelian groups, we need to find an analogue of ``maps to \(\R\)''.
We will see later that one way to do this is to consider full archimedean orders on \(G\).
Recall \Cref{example-abelian-orders} to see how this worked for \(G = \Z^2\).

Rephrasing \Cref{original-bns} in the language of partial orders yields \(\Phi \in \bns G\) if and only if the full subgraph of the Cayley graph spanned by \((\ker\Phi)_{\ordge}\) is connected, where \(\ordl\) is the order on \(G\) induced by \(\Phi\) and the standard order on \(\R\).
We would like to use this as a definition in cases where the codomain is any partially ordered group.
To make that work, we first need another notion of being connected.
\begin{Definition}
	Let \(G\) be a group and \(K \subseteq G\) a subset.
	We say that \(K\) is \emph{coarsely connected} if there exist a finitely generated subgroup \(H \subseteq G\) containing \(K\) and a finite generating set \(S \subseteq H\) such that the full subgraph of \(\Cay{H,S}\) spanned by \(K\) is connected.
\end{Definition}
\begin{Remark}\begin{enumerate}\item[]\label{coarsely-conn-indep-of-surroundings}
\item
	If \(G\) itself is finitely generated, then there is no need to pass to a subgroup \(H\) as every finite generating set of a subgroup is contained in a finite generating set of \(G\).

\item
	If \(G\) is contained in a finitely generated group \(G'\) such that \(K\) is coarsely connected as a subset of \(G'\), then \(K\) is also coarsely connected as a subset of \(G\).
	If \(K\) is coarsely connected as a subset of \(G\), then it is also coarsely connected as a subset of \(G'\).

\item
	Suppose \(G\) is finitely generated and \(S\) is a finite generating set.
	Then \(K \subseteq G\) is coarsely connected if and only if there exists some constant \(n \in \N\) such that for every \(g, h \in K\) there exists an \(n\)-path in \(\Cay{G,S}\) from \(g\) to \(h\) supported on \(K\).
	An \(n\)-path supported on \(K\) is a path in \(\Cay{G,S}\) such that any segment of \(n\) consecutive vertices on the path contains at least one point in \(K\).
	To see that this is equivalent to the definition of coarse connectedness, note that if \(K\) is connected by \(n\)-paths with respect to \(S\), then it is connected by \(1\)-paths with respect to the finite generating set that consists of words in \(S\) of length at most \(n\).
\end{enumerate}\end{Remark}
\begin{Example}
	In the group of integers, the subset of even numbers is coarsely connected as it is connected with respect to the generating set \(\{X, X^2\}\).
	The set of powers of \(2\) is not coarsely connected because the distance between two adjacent powers of \(2\) is unbounded.
	Similarly, for any surjective map \(\Phi\colon F_2 \onto \Z\), the preimage of the positive numbers is not coarsely connected.
	Recall \Cref{fig-cay-fg} for an illustration of the latter.
\end{Example}
In case \(G\) is not finitely generated, it might not be apparent why this is the correct definition.
But we will see in \Cref{nilpotent} how this definition ties in nicely with the finitely generated case.

Now we can define the \(\Sigma\)-invariant for maps to potentially non-abelian groups:

\begin{Definition}
	Let \(G\) be a group.
	Then the \emph{order \(\Sigma\)-invariant} \(\obns G\) is the subset of the set of non-trivial full archimedean orders on \(G\) defined as follows.

	Let \(\ordl\) be a non-trivial full archimedean order on \(G\).
	Then \(\ordl \in \obns G\) if and only if for every antichain normal subgroup \mbox{\(K \normsub G\)} that is maximal among antichain normal subgroups, \(K_{\ordge}\) is coarsely connected.
\end{Definition}

To understand how this definition generalises \Cref{original-bns}, we note that they align if we identify a character with the order it induces on \(G\).

\begin{Lemma}\label{maps-are-orders}
	Let \(G\) be a group and \(\Phi\colon G \to \R\) a character.
	Then \[\Phi \in \bns G \iff \indor\ordl\Phi \in \obns G,\]
	where \(\ordl\) is the standard order on \(\R\).
\end{Lemma}
\begin{proof}
	Take \(G\) to be ordered by \(\indor\ordl\Phi\).
	Then \(\ker\Phi\) is the only maximal antichain subgroup of \(G\) by \Cref{ker-max-ac}.
	Hence \(\indor\ordl\Phi \in \obns G\) if and only if \((\ker\Phi)_{\ordge}\) is coarsely connected.
	The latter is equivalent to \((\ker\Phi)_{\ordge}\) being connected for some finite generating set, which is the definition of \(\Phi \in \bns G\).
\end{proof}

Now that we have an analogue of \(\bns G\), we are missing just one ingredient to state a generalised version of \Cref{original-bns-iff-fg-abelian}:
The relative character sphere \(\relS G N\).

\begin{Definition}
	Let \(G\) be a group, \(N\) a normal subgroup and \(\pi\colon G \onto G/N\) the projection map.
	Then define the \emph{relative order sphere} as \[
		\orelS G N \coloneqq \{\indor\ordl\pi ~\mid~ \ordl ~\text{a non-trivial full archimedean order on}~ G/N\}.
	\]
\end{Definition}

Recall that in case \(G/N\) is abelian, \Cref{characterisation-abelian-orders} tells us that every full archimedean order on \(G/N\) is induced by a map to \(\R\).
\(\orelS G N\) contains all orders induced by maps to \(G/N\).
As \(G/N\) is abelian, every order on \(G/N\) is induced by a map to \(\R\).
That is \(\orelS G N\) contains precisely those orders induced by maps \(G \to G/N \to \R\), which are by definition exactly the maps in \(\relS G N\).
So this is actually a generalisation of \(\relS G N\) for \(G/N\) abelian in the same sense as \(\obns G\) is for \(\bns G\).

The following is an often useful description of the relative order sphere.
\begin{Lemma}\label{rel-sphere-all-pi}
	For a group \(G\) and \(N \normsub G\) a normal subgroup, \(\orelS G N\) is the set \[
		\{\ordl \mid \ordl \text{a non-trivial full archimedean order on}~ G ~\text{such that}~ N ~\text{is an antichain}\}.
	\]
\end{Lemma}
\begin{proof}
	By definition of \(\indor\ordl\pi\), for every order in \(\orelS G N\), \(N\) is an antichain.

	Now let \(\ordl\) be a full order such that \(N\) is an antichain.
	Then it is induced by the projection \(G \onto G/N\) and hence \(\ordl \in \orelS G N\).
\end{proof}

\section{Classification of partial orders on nilpotent groups}\label{nilpotent-orders}

By replacing \(\bns G\) and \(\relS G N\) in \Cref{original-bns-iff-fg-abelian} by \(\obns G\) and \(\orelS G N\), we obtain a criterion for finite generatedness of \(N\).
We will make this precise in \Cref{nilpotent} for the case \(G/N\) nilpotent.
But before we do that, we want to study what \(\orelS G N\) looks like in this case.
By definition, \(\orelS G N\) is in 1-to-1 correspondence with the non-trivial full archimedean orders on \(G/N\).
So our actual goal is to understand orders on nilpotent groups.
To begin, we recall the definition and some basic facts about nilpotent groups.
A comprehensive introduction may be found for example in \cite{CMZ}.

The trivial group is the only \emph{nilpotent group of class \(0\)}.
A \emph{nilpotent group of class \(n+1\)} is a group \(G\) such that \(G/Z(G)\) is nilpotent of class \(n\) but \(G\) itself is not.
Here and from now on, \(Z(G)\) denotes the \emph{center} of \(G\).
That is the subgroup of \(G\) containing all \(g\) such that \(g\) commutes with every other element of \(G\).
For example, if \(G\) is abelian, then \(Z(G) = G\) so \(G\) is nilpotent of class at most \(1\).

From a nilpotent group \(G\), we can derive what is called the \emph{lower central series} \[
	G_0 \supg G_1 \supg \dots \supg G_n.
\]
Its terms are \(G_0 = G\) and \[G_{i+1} = [G, G_i] = \bgrp{[g,h]}{g \in G, h \in G_i} \subg G.\]
By \([g,h]\) we mean the commutator of \(g\) and \(h\) and we adopt the convention \[[g,h] = \inv g \inv h gh.\]
Nilpotency of \(G\) guarantees that after finitely many steps, the trivial group appears as a term in the lower central series.
If \(n\) is the nilpotency class of \(G\), then \(G_n\) is the first trivial term.

We call \(G\) \emph{free-nilpotent of class \(n\) and rank \(k\)} if it has a \(k\)-element generating set and has no more relations than those absolutely necessary to make the group nilpotent.
To be precise, we get a presentation \[
	G = \grp{a_1\dots a_k}{G_n = 1}.
\]
Every nilpotent group of class no larger than \(n\) on at most \(k\) generators is a quotient of the respective free-nilpotent group.
This is analogue to free groups or free-abelian groups in the categories of groups or abelian groups.

Note that if \(S\) is a generating set of \(G\) and \(T\) is a generating set of \(G_i\), then \(G_{i+1}\) is generated by \(\{[g,h] \mid g \in S, h \in T\}\).
In particular, from a generating set \(S\) of \(G\), we may produce in a canonical way generating sets of every subgroup in the lower central series.

For any subgroup \(H \subg G\) we have that \(H\) is also nilpotent and \(H_i \subg G_i\).
In particular, the nilpotency class of \(H\) is no larger than the nilpotency class of \(G\).
If \(G\) is finitely generated, then so is \(H\).

We want to study the relation between \(\obns G\) and finiteness properties of the kernels of maps onto nilpotent groups.
In order to do that, we are first going to understand partial orders on nilpotent groups.

As we will see later, it will at least in this work be enough to consider cases where \(G\) is finitely generated and the order is full and archimedean.
We make these assumptions whenever they are convenient.
Let us start by looking at some examples.
\begin{Example}\label{example-nilpotent-orders}\begin{enumerate}\item[]
\item
	Recall \Cref{example-abelian-orders} constructing all full orders on finitely generated free-abelian groups.

\item
	Let \(H\) be the free-nilpotent group of class \(2\) and rank \(2\).
	It is also known as the \emph{(discrete) Heisenberg group}.
	To be specific, we have \[
		H = \bgrp{a,b}{1 = [a,[a,b]] = [b,[a,b]]}
	\]
	Any element \(g \in H\) can be written uniquely as \(g = a^\alpha b^\beta {\comm a b}^\gamma\) for some \(\alpha, \beta, \gamma \in \Z\).

	Any order on \[\Z^2 = H_\ab = H/\grpg{\comm a b}\] induces an order on \(H\) via the projection map.
	Even more, for any order on \(H\) the projection modulo \(\comm a b\) is order preserving.
	Thus any order is lexicographic with respect to the projection.
		
	In particular, if the order we pick on \(H_\ab\) is non-trivial, then the only way to complete this to an archimedean order on \(H\) is the order induced by the projection.

	So any archimedean order on \(H\) is either induced by the projection \mbox{\(H \onto H_\ab\)} or it is one of the two orders \[
		a^\alpha b^\beta {\comm a b}^\gamma \ordg 1 \iff \alpha = \beta = 0 \land \gamma > 0
	\] and \[
		a^\alpha b^\beta {\comm a b}^\gamma \ordg 1 \iff \alpha = \beta = 0 \land \gamma < 0
	\]
	That is, any order is lexicographic with respect to the exact sequence \[\grpg{\comm a b} \into H \onto H_\ab.\]
	To obtain an archimedean order, at least one of the two factors has to be ordered trivially by \Cref{total-archimedean-in-r}.

\item
	Now let \(G\) be the free-nilpotent group of class \(2\) and rank \(3\). That is \[
		G = \bgrp{a,b,c}{\comm x {\comm y z} = 1 ~\forall x,y,z \in \{a,b,c\}}
	\]
	It contains the Heisenberg group \(H\) as the subgroup generated by \(\{a, b\}\).
	Any order on \(G\) therefore restricts to an order on \(H\).
	If \(\restr{\ordl}H\) is induced by an order on \(H_\ab\), then \(\ordl\) is induced by an order on \(G_\ab\).
	If \(\restr{\ordl}H\) is one of the two archimedean orders such that \(\comm a b\) and \(1\) are comparable, then \(\ordl\) is induced by one of the inclusions \[
		\Z^2 = \grpg{\comm a b, c}_\ab \into G
	\] or \[
		\Z^3 \cong \grpg{\comm a b,\comm a c,\comm b c}_\ab = G_1 \into G.
	\]
	Up to choice of embedding of \(H\) and hence isomorphism of \(G\), these are all archimedean orders.
\end{enumerate}\end{Example}
For now, we omit the proof. At the end of this section we will have the tools to verify that these examples are indeed correct.

In the examples we see that the orders are largely determined by orders on \(G_\ab\).
For instance in the third example if there is more than a single generator comparable to \(1\), all of \(G_1\) is necessarily trivially ordered.
This leads to an intuition saying the more elements of \(G \setminus G_1\) are comparable to \(1\), the fewer possibilities there are to extend an order on \(G \setminus G_1\) to an order on \(G\).
With these ideas in mind, our goal is to make precise what orders on nilpotent groups look like.
The case of nilpotent groups of class \(1\), that is, abelian groups, has already been dealt with in \Cref{characterisation-abelian-orders}.

\begin{Remark}\label{abelian-all-generators-positive}
	Let \(G\) be a finitely generated free-abelian group and \(\Phi\colon G \to \R\).
	Then \(\ker\Phi\) is also finitely generated free-abelian.
	And so is \(G / \ker\Phi\), as any torsion element would have to be mapped by \(\Phi\) to some torsion element of \(\R\), but \(\R\) is torsion-free.

	Hence \(\ker\Phi \oplus G / \ker\Phi\) is also finitely generated free-abelian.
	By counting dimensions we see that it is even isomorphic to \(G\).

	Take \(\R\) to be ordered by the standard order.
	As \(\Phi\) factors through \(G / \ker\Phi\), \(\Phi\) induces a total order on \(G / \ker\Phi\) and the projection map then induces an order on \(G\).
	This is the same order that \(\Phi\) induces on \(G\).

	Any isomorphism \(G \cong \ker\Phi \oplus G / \ker\Phi\) produces a free-abelian generating set of \(G\) from such sets for \(\ker\Phi\) and \(G / \ker\Phi\).
	That is, for any full archimedean order on \(G\), if we make the right choice of free generators, every generator is either positive or incomparable to \(1\)
	and the positive cone is a subset of the subgroup spanned by the positive generators.
	This is not true for arbitrary nilpotent groups.
	For example in the Heisenberg group, it is possible that none of the generators is positive but their commutator is, as we have seen in \Cref{example-nilpotent-orders}.
\end{Remark}

Now let us consider the case where \(G\) is a non-abelian nilpotent group.
Any order on the abelianisation \(G_\ab\) induces an order on \(G\) and these orders we already understand by \Cref{characterisation-abelian-orders}.
They are the ones such that \(G_1\) is an antichain.
For all other orders, the following lemma is a restriction on how the order on \(G_1\) may look like.

\begin{Lemma}\label{commutator-infinitesimal}
	Let \(G\) be a finitely generated ordered nilpotent group and \(g, h \in G\) such that \(\comm g h \ordl g \ordg 1\).

	Then \(\comm g h \ordll g\).
\end{Lemma}
\begin{proof}
	Let \(N = \ngrpg g \normsub G\) be the normal subgroup generated by \(g\) and \(n\) the nilpotency class of \(N\).
	Note that \(\comm g h = \inv g (\inv h g h) \in N\).

	Suppose that every element of \(N_{i+1}\) is infinitesimal with respect to \(g\) and that \(\comm g h \in N_i\).
	Then for \(k \in \Z\) there is some \(c \in N_{i+1}\) (more precisely, it is some product of elements of the form \(\comm {x}{\comm g h}\), where \(x\) is again of this form or \(x=h\)) such that \[
		g{\comm g h}^{-2k} = h^{2k} g h^{-2k} c \ordg h^{2k} h^{-2k} c = c
	\] so \[
		g^2 \ordg \inv c g \ordg {\comm g h}^{2k}
	\] for any \(k \in \Z\).
	As \(\ordl\) is full, we get \(g \ordg {\comm g h}^k\) and hence \(g \ordgg \comm g h\).
	The claim follows by repeated application of this argument as \(N_n\) is trivial and therefore in particular infinitesimal with respect to \(g\).
\end{proof}

\begin{Theorem}\label{characterisation-nilpotent-orders}
	Let \(G\) be a finitely generated partially ordered nilpotent group with a full archimedean order \(\ordl\).
	Then there is a normal subgroup \(H \normsub G\) such that \(\ordl\) is induced by a total order on \(Z(G/H)\).

	That is, we may obtain any full archimedean order on \(G\) from the standard order on \(\R\) via the following chain of maps.
	\[
		\R \hookleftarrow \Z^n \cong Z(G/H) \into G/H \twoheadleftarrow G
	\]
\end{Theorem}
\begin{proof}
	Suppose that \(Z(G)\) is totally ordered.
	If \(1 \ordl g\) for some \(g \notin Z(G)\), we may choose some \(h \in G\) such that \(\comm g h \in Z(G)\).
	By \Cref{commutator-infinitesimal} we get \(\comm g h \ordll g\) contradicting \(\ordl\) being archimedean.
	Thus \(\poscone G \subseteq Z(G)\).
	That is, \(\ordl\) is induced by the inclusion \(Z(G) \into G\). So by letting \(H = 1\) we see that the claim is true.

	Otherwise, we do an induction. For this, we order finitely generated nilpotent groups as follows:
	Let \(A, B\) be finitely generated nilpotent groups.
	For any \(i \in \N\), \(A_i/A_{i+1}\) is a finitely generated free-abelian group.
	Set \(\rk_i A \coloneqq \rk A_i/A_{i+1}\).
	We say that \(A\) comes before \(B\) if and only if \(\rk_i A < \rk_i B\) for the largest \(i\) such that these ranks are not equal.

	If \(Z(G)\) is not totally ordered, pick a maximal cyclic subgroup \(C \subg Z(G)\) that is an antichain.
	As a subgroup of the center, \(C\) is automatically normal in \(G\).
	Since \(\ordl\) is full, it is induced by the projection \(G \onto G/C\).

	Let \(k\) be the largest number such that \(C \subg G_k\). Then we have \[
		\rk_i G/C = \begin{cases}
			\rk_i G & \text{if}~ i \neq k \\
			\rk_i G - 1 & \text{if}~ i = k
		\end{cases}
	\]
	Note that \(\rk_i G/C \subg \rk_i G\) for any \(i\) and \(rk_k G/C < \rk_k G\).
	Hence by induction as explained above, we find that there is some \(H' \normsub G/C\), such that \(\ordl\) is induced by the projection \(G \onto G/C \onto (G/C)/H'\) and the order on \((G/C)/H'\) is induced by a total order on its center.

	By \Cref{inducing-maps-chain}, \(\ordl\) is then induced by the projection \(G \onto (G/C)/H'\) and setting \(H\) to be the kernel of this projection finishes the proof.
\end{proof}
Conversely, every choice of \(H\) such that \(G/H\) is torsion free and \(\iota\colon Z(G/H) \into \R\) induces a unique order on \(G\).
Two choices \(H, \iota\) and \(H', \iota'\) yield the same order if and only if \(H = H'\) and \(\iota = \lambda\iota'\) for some \(\lambda \in \R^+\).
Thus we get a full characterisation of all full archimedean orders on finitely generated nilpotent groups.

\begin{Remark}\label{abelian-orders-are-maps}
	Note that if \(G\) is abelian, \(G/H = Z(G/H)\), so we get that \(\ordl\) is induced by \[
		\R \hookleftarrow \Z^n \cong Z(G/H) \cong G/H \twoheadleftarrow G.
	\]
	Hence we recover that every full archimedean order on \(G\) is induced by a map to \(\R\).
\end{Remark}

To conclude the section, now is a good time to revisit \Cref{example-nilpotent-orders}.
The center of the Heisenberg group is \[ Z(H) = \grpg{\comm a b} \cong \Z. \]
Let \(\ordl\) be an order on \(H\).
By \Cref{characterisation-nilpotent-orders}, we know that there is \(P \normsub H\) such that \(\ordl\) is induced by a total order on \(Z(H/P)\).
If \(P\) is trivial, then \(Z(H)\) is totally ordered by one of the two total orders on \(\Z\).
Spelling this out, we obtain one of the two orders of the form \[
	a^\alpha b^\beta {\comm a b}^\gamma \ordg 1 \iff \alpha = \beta = 0 \land \gamma \ordg 1.
\]
If \(P\) is non-trivial it contains \(Z(H)\), so \(\ordl\) is induced by an order on some quotient of \(H / Z(H) = H_\ab \cong \Z^2\).
In particular, it is also induced by an order on \(\Z^2\), namely the one which itself is induced by the projection onto said quotient.

The case of the free-nilpotent group of class \(2\) and rank \(3\) may be handled similarly by looking at all possible intersections of \(P\) and \(Z(G)\).

\section{Maps onto nilpotent groups}\label{nilpotent}

Let \(G\) be a finitely generated group, \(Q\) nilpotent and \(\Phi\colon G \onto Q\) onto.
We would like to know if \(\ker\Phi\) is also finitely generated.

For \(Q\) nilpotent of class \(1\), that is to say \(Q\) abelian, recall \Cref{original-bns-iff-fg-abelian}.
It states that \(\ker\Phi\) is finitely generated if and only if \(\relS G{\ker\Phi} \subseteq \bns G\).
By \Cref{maps-are-orders} and \Cref{abelian-orders-are-maps}, the theorem remains true if we replace \(\relS G{\ker\Phi} \subseteq \bns G\) by \(\orelS G{\ker\Phi} \subseteq \obns G\).
This replacement also makes it possible to at least state the theorem if \(G/\ker\Phi\) is any group.
The goal of this section is to prove it in case \(G/\ker\Phi\) is nilpotent.

As we know from \Cref{characterisation-nilpotent-orders} that all orders on nilpotent groups are induced by the inclusion of the center into some quotient, let us investigate how \(\orelS G N\) and \(\obns G\) behave when passing to subgroups.

\begin{Lemma}\label{rel-sphere-pass-to-subgroup}
	Let \(G\) be a group, \(H \subg G\) a subgroup and \(\Phi\colon G \onto Q\) a map onto some group \(Q\) such that \(\ker\Phi \normsub H\).

	Let \(\ordl_G~ \in \orelS G {\ker\Phi}\). It is induced by some order \(\ordl_Q\) on \(Q\).
	Suppose that \(\ordl_Q\) is induced by the inclusion \(\Phi(H) \into Q\).
	\(\restr{\ordl_Q}{\Phi(H)}\) induces via \(\inv\Phi\) an order \(\ordl_H\) on \(H\) and \(\ordl_H~ \in \orelS H {\ker\Phi}\).

	Then \(\ordl_G\) is induced by \(\ordl_H\) via the inclusion \(H \into G\).
	This statement is also visualized in \Cref{rel-sphere-subgroup-diagram}.
	\begin{figure}
		\centering
\begin{tikzpicture}[shorten >=1pt,on grid,auto]

	\node[] (H) {\(H\)};
	\node[] (G) [right=3 of H] {\(G\)};
	\node[] (P) [below=1.5 of H] {\(\Phi(H)\)};
	\node[] (Q) [right=3 of P] {\(Q\)};
	\path[->>]
		(H) edge (P)
		(G) edge (Q)
	;
	\path[right hook ->]
		(H) edge [dashed] (G)
		(P) edge (Q)
	;
\end{tikzpicture}
		\caption{\(H \into G\) is order-inducing if all the other maps are.}
		\label{rel-sphere-subgroup-diagram}
	\end{figure}
\end{Lemma}
\begin{proof}
	Recall \Cref{order-from-positive} stating that an order is characterised entirely by its positive cone.

	Let \(1 \ordl_G g \in G\).
	Then \(1 \ordl_Q \Phi(g) \in Q\).
	So \(\Phi(g) \in \Phi(H)\).
	Hence there is some \(h \in H\) such that \(\Phi(h) = \Phi(g)\).
	That is, \(g \inv h \in \ker\Phi \normsub H\), so \mbox{\(g = (g \inv h)h \in H\)}.
\end{proof}

\begin{Lemma}\label{bns-pass-to-subgroup}
	Let \(G\) be a group and \(Q\) a finitely generated nilpotent group.
	Let \mbox{\(\Phi\colon G \onto Q\)} be onto.

	Take \(Q\) to be ordered by \(\ordl_Q\) and \(G\) ordered by \(\indor {\ordl_Q} \Phi\).
	Let \(P\) be the subgroup of \(Q\) such that \(\ordl_Q\) is induced by a total order on \(Z(Q/P)\) as provided by \Cref{characterisation-nilpotent-orders}.
	The order on \(Q/P\) that induces \(\ordl_Q\) is called \(\ordl\).
	Set \(\Psi \coloneqq \pi \circ \Phi\) and \(H \coloneqq \inv\Psi(Z(Q/P))\).
	This situation is summed up in \Cref{bns-pass-to-subgroup-diagram}.
	\begin{figure}
		\centering
\begin{tikzpicture}[shorten >=1pt,on grid,auto]
	\node[] (H) {\(H\)};
	\node[] (G) [right=3 of H] {\(G\)};
	\node[] (Q) [below=1.5 of G] {\(Q\)};
	\node[] (QP) [below=1.5 of Q] {\(Q/P\)};
	\node[] (piZQP) [below=1.5 of H] {\(\inv\pi(Z(Q/P))\)};
	\node[] (ZQP) [below=1.5 of piZQP] {\(Z(Q/P)\)};
	\path[->>]
		(G) edge [right] node {\(\Phi\)} (Q)
		(Q) edge [right] node {\(\pi\)} (QP)
		(H) edge [right] node {\(\Phi\)} (piZQP)
		(piZQP) edge [right] node {\(\pi\)} (ZQP)
	;
	\path[right hook ->]
		(H) edge (G)
		(ZQP) edge (QP)
		(piZQP) edge (Q)
	;
\end{tikzpicture}
		\caption{If we start with an ordered nilpotent group \(Q\) and a map \(\Phi\), we get induced orders on all the groups in this diagram.}
		\label{bns-pass-to-subgroup-diagram}
	\end{figure}

	Then \[\indor\ordl\Psi \in \obns G \iff \indor{\restr\ordl{Z(Q/P)}}\Psi \in \obns H.\]
\end{Lemma}
\begin{proof}
	By \Cref{ker-max-ac}, the only maximal antichain subgroup of \(H\) is \(\ker\Psi\).
	Also, any maximal antichain normal subgroup of \(G\) gets mapped by \(\Psi\) to a maximal antichain normal subgroup of \(Q/P\).
	Since \(Q/P\) is nilpotent, every nontrivial normal subgroup of \(Q/P\) has nontrivial intersection with \(Z(Q/P)\).
	Thus the only normal antichain subgroup of \(Q/P\) is the trivial group.
	Hence the only maximal normal antichain subgroup of \(G\) is \(\ker\Psi\).
	The claim follows from \Cref{coarsely-conn-indep-of-surroundings}.
\end{proof}

\begin{Theorem}\label{bns-iff-fg}
	Let \(G\) be a finitely generated group, \(Q\) finitely generated nilpotent and \(\Phi\colon G \onto Q\) onto.
	The following are equivalent.
	\begin{enumerate}
		\item\label{bns-iff-fg-fg} \(\ker\Phi\) is finitely generated
		\item\label{bns-iff-fg-bns} \(\orelS G {\ker\Phi} \subseteq \obns G\)
		\item\label{bns-iff-fg-bnsz} \(\orelS {\bar Z} {\ker\Phi} \subseteq \obns{\bar Z}\)
	\end{enumerate}
	where \(\bar Z \coloneqq \inv\Phi(Z(Q))\).
\end{Theorem}
\begin{proof}
\(\ref{bns-iff-fg-fg} \Rightarrow \ref{bns-iff-fg-bns}\):
	Let \(\ordl \in \orelS G {\ker\Phi}\).
	By definition, \(\ordl\) is induced by a unique order on \(Q\).
	\Cref{characterisation-nilpotent-orders} tells us that there is a normal subgroup \(P \normsub Q\) such that the order on \(Q\) is induced by a total order on \(Z(Q/P)\).

	Let \(\pi\) be the projection map \(Q \onto Q/P\) and set \[H \coloneqq \inv{(\pi \circ \Phi)}(Z(Q/P)) \subseteq G.\]
	We know that \[g \ordg 1 \iff \Phi(g)P \ordg 1 \in Z(Q/P),\] so \(\restr\ordl H \in \obns H\) implies \(\ordl \in \obns G\) by \Cref{bns-pass-to-subgroup}.
	Hence it is enough to show that \(\restr \ordl H \in \obns H\).

	\(\restr{(\pi \circ \Phi)}H \colon H \onto Z(Q/P)\) is a surjective map onto an abelian group.
	So by \Cref{original-bns-iff-fg-abelian}, if its kernel is finitely generated, then \(\restr\ordl H \in \obns H\).
	By construction, \(\ker(\pi \circ \Phi) \normsub H\), so \(\ker(\pi \circ \Phi) = \ker\restr{(\pi \circ \Phi)}H\).

	Note that we may write \(\ker(\pi \circ \Phi)\) as an extension \[
		\ker\Phi \into \ker(\pi \circ \Phi) \onto \ker\pi.
	\] Using different names for the same groups, we obtain the extension \[
		\ker\Phi \into \ker\restr{(\pi \circ \Phi)}H \onto P.
	\]
	By assumption, \(\ker\Phi\) was finitely generated.
	As \(P\) is a subgroup of a finitely generated nilpotent group and hence itself finitely generated, \(\ker\restr{(\pi \circ \Phi)}H\) is also finitely generated.
	Hence \(\restr\ordl H \in \obns H\) and \(\ordl \in \obns G\).

\(\ref{bns-iff-fg-bns} \Rightarrow \ref{bns-iff-fg-bnsz}\):
	Let \(\ordl \in \orelS{\bar Z} {\ker\Phi}\).
	Then \(\ordl\) is induced by an order on \(\bar Z/\ker\Phi \cong Z(Q)\).
	This order induces an order on \(Q\) and that order induces an order \(\ordl'\) on \(G\).

	\(\ordl'\) is hence induced by the inclusion \(\bar Z \into G\) by \Cref{rel-sphere-pass-to-subgroup} and \[\restr{\ordl'}{\bar Z} = \ordl.\]
	Thus if \(\ordl' \in \obns G\), then \(\ordl \in \obns{\bar Z}\) by \Cref{bns-pass-to-subgroup}.
	As \(\ker\Phi\) is an antichain with respect to \(\ordl\)
	and \(\ker\Phi \normsub \bar Z\), the kernel is also an antichain with respect to \(\ordl'\).
	Hence \[
		\ordl' \in \orelS G{\ker\Phi} \subseteq \obns G
	\] and therefore \(
		\ordl \in \obns{\bar Z}.
	\)

\(\ref{bns-iff-fg-bnsz} \Rightarrow \ref{bns-iff-fg-fg}\):
	Let \(H\) be a finitely generated subgroup of \(\bar Z\) containing \(\ker\Phi\).
	If no such \(H\) exists, then \((\ker\Phi)_{\ordge}\) is not coarsely connected for any order on \(\bar Z\).
	That is no order in \(\orelS {\bar Z}{\ker\Phi}\) is contained in \(\obns {\bar Z}\).
	As the former contains one element for every order on \(\bar Z / \ker\Phi = Z(Q)\) and \(Z(Q)\) is a nontrivial torsion-free abelian group, the relative order sphere is in particular nonempty.
	Hence if \ref{bns-iff-fg-bnsz} is true, \(H\) must exist.

	We have \(\orelS H {\ker\Phi} \subseteq \orelS {\bar Z} {\ker\Phi} \subseteq \obns {\bar Z}\).
	By \Cref{bns-pass-to-subgroup}, this also means \(\orelS H {\ker\Phi} \subseteq \obns H\).
	As \(H/\ker\Phi\) is a subgroup of \(Z(Q)\) and hence abelian, \Cref{original-bns-iff-fg-abelian} shows that \(\ker\Phi\) is finitely generated, finishing the proof.

\end{proof}

\end{document}